\documentclass[11pt,reqno]{amsart}
\usepackage[utf8]{inputenc}
\usepackage{amsmath}
\usepackage{amsfonts}
\usepackage{amssymb}
\usepackage{amsthm}
\usepackage{hyperref}
\usepackage{multicol}
\usepackage{pictexwd, dcpic}
\usepackage{fancyhdr}
\usepackage{mathrsfs}
\usepackage{latexsym,mathtools}
\usepackage{bm, enumerate}
\usepackage{graphicx}
\usepackage{color,soul}
\usepackage{xcolor}
\usepackage{hyperref}
\usepackage{dsfont}
\usepackage{todonotes}
\usepackage{tikz-cd}
\makeatletter
\@namedef{subjclassname@2020}{%
	\textup{2020} Mathematics Subject Classification}
\makeatother

\newcommand{\beq}{\begin{eqnarray}}
	\newcommand{\eeq}{\end{eqnarray}}
\newcommand{\beqn}{\begin{eqnarray*}}
	\newcommand{\eeqn}{\end{eqnarray*}}
\newcommand{\rar}{\rightarrow}

\reversemarginpar  % These 2 commands work very fine

\numberwithin{equation}{section}

\newtheorem{theorem}{\bf Theorem}[section]
\newtheorem{proposition}[theorem]{\bf Proposition}%[section]
\newtheorem{corollary}[theorem]{\bf Corollary}%[section]
\newtheorem{lemma}[theorem]{\bf Lemma}%[section]

\theoremstyle{remark}
\newtheorem{example}[theorem]{\bf Example}
\newcommand*{\Ge}{\geqslant}

\newcommand*{\inp}[2]{\langle{#1},\,{#2} \rangle}
\newcommand*{\Le}{\leqslant}

\theoremstyle{definition}

\theoremstyle{remark}

\hypersetup{
	colorlinks,
	citecolor=blue,
	linkcolor=blue,
	urlcolor=blue}

\title{}
\begin{document}

% \title[Norm formula and Sarason's containment result for finite rank de Branges-Rovnyak space]{Norm formula and Sarason's containment result for finite rank de Branges-Rovnyak space}
% \title[A norm formula for finite rank de Branges Rovnyak spaces and its applications]{A norm formula for finite rank de Branges Rovnyak spaces and its applications}
\title[Unbounded Toeplitz operators and finite rank de Branges-Rovnyak spaces]{Unbounded Toeplitz operators and finite rank de Branges-Rovnyak spaces}

\author[S. Ghara]{Soumitra Ghara}
\address{Department of Mathematics\\
Indian Institute of Technology  Kharagpur, Midnapore-721302, India}
   \email{soumitra@maths.iitkgp.ac.in\\ ghara90@gmail.com}

\author[R. Reza]{MD Ramiz Reza}
\address{School of Mathematics, Indian Institute of Science Education and Research Thiruvananthapuram, Kerala-695551, India}
\email{ramiz.md@gmail.com, ramiz@iisertvm.ac.in}

\author[C. K. Sahu]{Chaman Kumar Sahu}
\address{Department of Mathematics, Indian Institute of Technology Bombay,
Powai, Mumbai, 400076, India}
\email{sahuchaman9@gmail.com, chamanks@math.iitb.ac.in}

\subjclass[2020]{}

\begin{abstract}
Motivated by the recent developments of de Branges-Rovnyak spaces, we investigate the function theoretic aspects of finite rank de Branges-Rovnyak spaces $H(B)$ generated by row-valued Schur functions $B$. We provide a  generalization of Sarason's fundamental work by characterizing finite rank $H(B)$-spaces as the domain of the adjoint of the Toeplitz operators $T_\varphi^*$ with symbol $\varphi = BA^{-1}$, where $A$ is an matrix-valued outer function satisfying $A^*A+B^*B = I$ a.e. on the unit circle. We derive a norm formula for functions in $H(B)$-space and provide a concrete realization of this norm in terms of the Taylor coefficients of the function and the symbol $\varphi$. As an application, we characterize  all symbols $B$ for which $H^\infty \subseteq H(B)$ in terms of the boundary behavior of $I-BB^*$, thereby extending Sarason's criterion for the classical de Branges–Rovnyak spaces.
% we characterize 
% the continuous embedding $H^\infty \subseteq H(B)$ via the boundary behavior of $I-BB^*$, thereby extending Sarason's criterion for the classical de Branges–Rovnyak spaces.
% provide a characterization for the containment of the algebra of bounded analytic functions on the unit disc inside $H(B)$ in terms of the boundary behavior of $I-BB^*$, generalizing Sarason's criterion for the classical de Branges–Rovnyak spaces.
\end{abstract}

\maketitle

\section{Introduction}

Let $\mathbb N$ and $\mathbb C$ denote the sets of positive integers and complex numbers, respectively. Let $\mathbb{D} = \{z \in \mathbb{C} : |z| < 1\}$ and $\mathbb{T} = \{z \in \mathbb{C} : |z| = 1\}$ denote the open unit disc and the unit circle, respectively. For complex separable Hilbert spaces $\mathcal X$ and $\mathcal Y$, let $\mathcal L(\mathcal X, \mathcal Y)$ denote the space of bounded linear operators from $\mathcal X$ into $\mathcal Y$. The \emph{Schur class} $\mathcal S(\mathcal X, \mathcal Y)$ consists of all holomorphic functions $f : \mathbb D \to \mathcal L(\mathcal X, \mathcal Y)$ such that
$$\sup_{z \in \mathbb D} \|f(z)\|_{\mathcal L(\mathcal X, \mathcal Y)} \leqslant  1.$$

% Let $\mathbb N$ and $\mathbb C$ denote the set of all positive integers and complex numbers, respectively. The open unit disc $\{z \in \mathbb C : |z| < 1\}$ is denoted by $\mathbb D$. The unit circle $\{z \in \mathbb C: |z| = 1\}$ is denoted by $\mathbb T$. For complex separable Hilbert spaces $\mathcal X$ and $\mathcal Y$, let $\mathcal L(\mathcal X, \mathcal Y)$ denote the algebra of bounded linear operators from $\mathcal X$ to $\mathcal Y$, and the {\it Schur class} $\mathcal{S}(\mathcal{X}, \mathcal{Y})$ is defined as the set of all operator-valued holomorphic functions mapping  $\mathbb D$ to the closed unit ball of bounded linear operators, that is,
% $$\mathcal S(\mathcal X, \mathcal Y) = \Big\{ f: \mathbb{D} \to \mathcal {L}(\mathcal {X}, \mathcal {Y}) \mid f \text{ is holomorphic and } \sup_{z \in \mathbb {D}} \|f(z)\| \le 1 \Big\}.$$ 

For $B \in \mathcal S(\mathcal X, \mathcal Y)$, the associated \emph{de Branges--Rovnyak space} $H(B)$ is the reproducing kernel Hilbert space with the $\mathcal L(\mathcal Y, \mathcal Y)$-valued kernel
  
\beqn
K_B(z,\lambda) = \frac{I-B(z)B(\lambda)^*}{1-z\overline{\lambda}},\qquad z, \lambda \in \mathbb D.
\eeqn
De Branges-Rovnyak space $H(B)$ can equivalently be realized as the range space of $(I-T_BT_B^*)^{1/2}$ equipped with the range norm, where $T_B$ is the Toeplitz operator induced by the symbols  $B \in \mathcal S(\mathcal X, \mathcal Y),$ see \cite{Ros-Rov}. In case, $\mathcal X= \mathcal Y = \mathbb C$, we refer to $H(B)$ as the classical de Branges-Rovnyak space and it is denoted by $H(b)$. An important case arises when $B \equiv 0$, in which case $H(B)$ coincides with the Hardy space $H^2_{\mathcal Y}$ of $\mathcal Y$-valued analytic functions on $\mathbb D$ with square-summable Taylor coefficients. When $\mathcal Y= \mathbb C$, we write $H^2$ for $H^2_{\mathbb C}$. 

% de Branges–Rovnyak spaces have emerged as a rich framework in modern operator theory and function theory, with deep connections to model theory of contractions, interpolation problems, and spaces of analytic functions.
Originating in the work of de Branges and Rovnyak and subsequently developed in a systematic way by Donald Sarason, de Branges–Rovnyak spaces provide a flexible rich framework that unifies function-theoretic and operator-theoretic techniques, see \cite{Bran-Rovn, Sa1,Sa2,Sa3} and references therein. In recent years, it has found its renewed importance in areas such as control theory, structured operator models, and multivariable operator theory, particularly through their role as reproducing kernel Hilbert spaces associated with contractive analytic functions, see for example \cite{Chu, Co-Ra, CGR, CGR-2022, FM-1, FM-2,LGR, Sa4} and references therein.

%  The reproducing kernel of $H^2_{\mathcal Y}$ is
% \[
% K_{\mathbf 0}(z,\lambda) = \frac{1}{1 - z\overline{\lambda}} I_{\mathcal Y}, \qquad z, \lambda \in \mathbb D.
% \]

In this article, we study the \emph{finite rank de Branges--Rovnyak spaces} $H(B),$ which are generated by the Schur functions $B = (b_1, b_2, \ldots, b_n) \in \mathcal S(\mathbb C^n,\mathbb C)$ such that $\{b_1,b_2,\ldots,b_n\}$ is a linearly independent set, see \cite[Definition~1.1]{A-M}. The kernel then takes the form
\[
K_B(z,\lambda) = \frac{1-\sum_{j=1}^n b_j(z)\overline{b_j(\lambda)}}{1-z\overline{\lambda}}, \quad z, \lambda \in \mathbb D.
\]
While several aspects of the scalar theory extend formally to this setting, new challenges arise due to the operator-valued nature of the symbol, see \cite{A-M}. In particular, obtaining concrete realizations of $H(B)$, deriving computable norm formulas, and understanding embeddings of classical function spaces into $H(B)$ require special attention. 
Under the condition $\log(1-BB^*) \in L^1(\mathbb T)$, it is well known that $H(B)$ is invariant under the shift operator $M_z,$ see \cite{A-M}. In this case, there exist outer functions $a\in  \mathcal S(\mathbb C,\mathbb C)$ and $A\in  \mathcal S(\mathbb C^n,\mathbb C^n)$ satisfying
\begin{equation}\label{scalar outer a}
\sum_{j=1}^n |b_j(\zeta)|^2 + |a(\zeta)|^2 = 1 \quad \text{a.e. } \zeta \in \mathbb T,
\end{equation}
and
\begin{equation}\label{unique-A}
B(\zeta)^*B(\zeta) + A(\zeta)^*A(\zeta) = I \quad \text{a.e. } \zeta \in \mathbb T.
\end{equation}
Also, under the aforementioned assumption,  
\begin{align}\label{f^+ toeplitz expression}
  \text{a function $f\in H^2$ is contained in $H(B)$ $\iff$ $\exists$  $f^+\in H^2_{\mathbb C^n}$ such that $T_B^*f+T_{A}^*f^+=0$.}
\end{align}
Moreover, in this case, $f^+$ is unique and 
\beq
\label{norm-formula}
\|f\|_{H(B)}^2 = \|f\|_{H^2}^2 + \|f^+\|_{H^2_{\mathbb C^n}}^2, \quad f \in H(B),
\eeq
see \cite[Proposition~5.1 \& Theorem~5.2]{A-M} and a more general result is provided in \cite[Theorem~2.2(i)]{CGL}.
Since $A$ is outer, it is invertible on $\mathbb D$, see \cite[pg. 22]{Nik}, and we set $\varphi = BA^{-1}$.  In the scalar case, Sarason identifies $H(b)$ as the domain of the adjoint of a (possibly unbounded) Toeplitz operator $T_{\varphi}$. Our first main result below extends this to finite rank de Branges-Rovnyak spaces $H(B),$ thereby generalizing \cite[Proposition~5.4]{Sa1}.

% $A^{-1}= \frac{adj(A)}{det A}: \mathbb D \rar \mathcal L(\mathbb C^n, \mathbb C^n)$ is a holomorphic function in the Smirnov class $N^{+}(\mathcal L(\mathbb C^n, \mathbb C^n)).$

 \vspace{5pt}
\textbf{Theorem~A.}(Theorem~\ref{domain-debranges} below) Let $B \in \mathcal S(\mathbb C^n, \mathbb C)$ be such that $\log(1-BB^*) \in L^1(\mathbb T).$ Let $A: \mathbb D \rar \mathcal L(\mathbb C^n, \mathbb C^n)$ be an outer function satisfying \eqref{unique-A}, and 
let $\varphi = BA^{-1}: \mathbb D \rar \mathcal L(\mathbb C^n, \mathbb C).$ Then $T_{\varphi}$ is a densely defined operator in $H^2_{\mathbb C^n}$, and  the domain of $T_\varphi^*$ is $H(B)$. Moreover, $T_\varphi^*f = -f^+$ for $f \in H(B)$, and consequently, 
$$
\|f\|_{H(B)}^2 = \|f\|_{H^2}^2 + \|T_\varphi^*f\|_{H^2_{\mathbb C^n}}^2, \quad f \in H(B).
$$

% \vspace{10pt}
% . 

Note that Theorem~A establishes a norm formula in terms of an operator-theoretic identity whose utility is significantly enhanced when the action of $T_\varphi^*$ can be explicitly computed. The next theorem provides exactly this for functions holomorphic in a neighborhood of $\overline{\mathbb D}$, yielding a computable formula for the $H(B)$ norm in terms of its Taylor coefficients.  This generalizes the classical case \cite[Theorem~4.1]{CGR} for finite rank de Branges-Rovnyak spaces. 

 \vspace{5pt}

\textbf{Theorem~B.} (Theorem~\ref{norm-formula-thm} below) Let $B \in \mathcal S(\mathbb C^n, \mathbb C)$ be such that $\log(1-BB^*) \in L^1(\mathbb T).$ Let $A: \mathbb D \rar \mathcal L(\mathbb C^n, \mathbb C^n)$ be an outer function satisfying \eqref{unique-A}.
Let $BA^{-1}: \mathbb D \rar \mathcal L(\mathbb C^n, \mathbb C)$ admit the series expansion $\sum_{j=0}^\infty c_jz^j$. Suppose $f$ is holomorphic in a neighborhood of $\overline{\mathbb D}$, say, $f(z) = \sum_{j=0}^\infty \hat{f}(j) z^j$. Then the series $\sum_{j=0}^\infty \hat{f}(j+k){c_j^*}$ converges absolutely for each integer $k \Ge 0$, and 
\beqn
% \label{norm-formula-nbd}
\|f\|_{H(B)}^2 = \sum_{k=0}^\infty |\hat{f}(k)|^2 + \sum_{k=0}^\infty \Big\|\sum_{j=0}^{\infty} c_j^* \hat{f}(j+k)\Big\|^2.
\eeqn

As a consequence of Theorem~B, we obtain explicit formulas for the norms of monomials, the S\"{z}ego kernel functions, and various other functions in $H(B)$. In the classical setting, such computations play a central role in understanding when a de Branges--Rovnyak space $H(b)$ coincides with weighted Dirichlet-type spaces
% , as well as in determining conditions under which their norms are equal or equivalent
(see, e.g., \cite{Co-Ra, CGR, Sa3}).

A common approach to analyzing the structure of Hilbert spaces of analytic functions is to study their embeddings into well-understood function spaces. In this direction, a classical result of Sarason \cite[Theorem~1]{Sa2} characterizes when $H(b)$ contains $H^\infty$, the algebra of bounded analytic functions on $\mathbb D$. Our next result extends this characterization to the finite rank setting.

\vspace{5pt}
\textbf{Theorem~C.}(Theorem~\ref{bounded-contained} below)
    Let  $B  \in \mathcal S(\mathbb C^n, \mathbb C)$ be such that $\log(1-BB^*) \in L^1(\mathbb T).$ Let $a$ and $A$ be outer functions satisfying \eqref{scalar outer a} and \eqref{unique-A}, respectively. Then, the following are equivalent:
    \begin{itemize}
        \item [(i)] $H^\infty \subseteq H(B)$,
        \item [(ii)] $\sup_{n \Ge 0} \|z^n\|_{H(B)} < \infty$,
        \item [(iii)] $BA^{-1} \in H^2_{\mathcal L(\mathbb C^n, \mathbb C)}$,
        \item[(iv)]$\frac{B}{a}\in H^2_{\mathcal L(\mathbb C^n, \mathbb C)}$,
        \item [(v)] $(1 - BB^*)^{-1} \in L^{1}(\mathbb T)$.
        \end{itemize}

% Let  $B \in \mathcal S(\mathbb C^n, \mathbb C),$ and $B(z) = (b_1(z), b_2(z),\ldots, b_n(z)),~ z\in \mathbb D.$ Therefore $B: \mathbb D \rar \mathcal L(\mathbb C^n, \mathbb C)$ is an analytic function  satisfies $\sum_{j=1}^n|b_j(z)|^2 \Le 1,~ z \in \mathbb D.$ 

% We denote by $H(B),$ the reproducing kernel Hilbert space associated to the kernel function $K_B$ defined by 
% \beqn
% K_B(z, w) = \frac{1-\sum_{i=1}^{n} b_i(z)\overline{b_i(w)}}{1-z\overline{w}}, \quad z, w \in \mathbb D.
% \eeqn
The content of this paper is organized as follows. In the next section, we show that the finite rank de Branges-Rovnyak space is the domain of adjoint of (possibly unbounded) Toeplitz operators (Theorem~\ref{domain-debranges}). As a result, we prove a norm formula of functions of finite rank de Branges-Rovnyak spaces (Theorem~\ref{norm-formula-thm}). In section~\ref{3}, we obtain a norm formula of several functions in $H(B)$ including monomials and the S\"{z}ego kernel functions, illustrating the results of Section~\ref{2} (Corollary~\ref{monomial-norm}). We also present an example of a Schur class function $B$ with explicit computation of matrix-valued outer function and give norm formula of several functions in the associated de Branges-Rovnyak space (Example~\ref{explicit-ex1}). In Section~\ref{4}, we characterize $H(B)$-spaces which contain the Banach algebra $H^\infty$ of bounded analytic functions on $\mathbb D$ (Theorem~\ref{bounded-contained}), and end the paper with some concluding remarks.

\section{$H(B)$-spaces and unbounded Toeplitz operators}\label{2}
 Let $\mathcal X$ and $\mathcal Y$ be complex separable Hilbert spaces. We write $L^2(\mathbb T,\mathcal Y)$ for the Hilbert space of (equivalence classes of) square-integrable $\mathcal Y$-valued functions on $\mathbb T$. For the case $\mathcal Y = \mathbb C$, we denote this space by $L^2(\mathbb T)$. Recall that $H^2_{\mathcal Y}$ is the $\mathcal Y$-valued Hardy space given by 
 \[
H^2_{\mathcal Y} := \left\{ \sum_{n=0}^{\infty} a_n z^n : a_n \in \mathcal Y,\ \sum_{n=0}^{\infty} \|a_n\|_{\mathcal Y}^2 < \infty \right\}.
\]
 We can view $H^2_{\mathcal Y}$ as a closed subspace of $L^2(\mathbb T, \mathcal Y)$ in the usual manner by considering the boundary values of the analytic functions in $H^2_{\mathcal Y}$. Let $P : L^2(\mathbb T,\mathcal Y) \to H^2_{\mathcal Y}$ be the orthogonal projection. For a measurable function $\psi : \mathbb T \to \mathcal L(\mathcal X,\mathcal Y)$, the (operator-valued) Toeplitz operator $T_\psi$ is formally defined on $H^2_{\mathcal X}$ by
\[
T_\psi f = P(\psi f).
\]
In case, if $\psi \in L^\infty(\mathbb T,\mathcal L(\mathcal X,\mathcal Y))$, which is the space of $\mathcal L(\mathcal X,\mathcal Y)$-valued essentially bounded measurable functions on $\mathbb T$, then $T_{\psi}$ defines a bounded operator from $H^2_{\mathcal X}$ to $H^2_{\mathcal Y}$.  In general, if $\psi$ is not essentially bounded, $T_\psi$ need not be bounded, and is understood as a densely defined operator with domain
\[
\mathcal{D}(T_\psi) = \{ f \in H^2_{\mathcal X} : \psi f \in L^2(\mathbb T,\mathcal Y) \}.
\]
Such operators are referred to as (possibly unbounded) Toeplitz operators. They arise naturally in the study of de Branges-Rovnyak spaces and play a key role in the operator-theoretic description of these spaces. 

In this section, our first main result is below, showing the relationship between $H(B)$-spaces and unbounded Toeplitz operators with vector-valued symbols. 
 
\begin{theorem}\label{domain-debranges}
Let $B \in \mathcal S(\mathbb C^n, \mathbb C)$ be such that $\log(1-BB^*) \in L^1(\mathbb T).$ Let $A: \mathbb D \rar \mathcal L(\mathbb C^n, \mathbb C^n)$ be an outer function satisfying \eqref{unique-A}, and 
let $\varphi = BA^{-1}: \mathbb D \rar \mathcal L(\mathbb C^n, \mathbb C).$ Then $T_{\varphi}$ is a densely defined operator in $H^2_{\mathbb C^n}$, and  the domain of $T_\varphi^*$ is $H(B)$. Moreover, $T_\varphi^*f = -f^+$ for $f \in H(B)$, and consequently, 
$$
\|f\|_{H(B)}^2 = \|f\|_{H^2}^2 + \|T_\varphi^*f\|_{H^2_{\mathbb C^n}}^2, \quad f \in H(B).
$$
%    Let $B: \mathbb D  \rar \mathcal L(\mathbb C^n, \mathbb C)$ be a contractive holomorphic function $($that is, $B = (b_1, b_2, \ldots, b_n)$ is a $\mathbb C^n$-valued holomorphic function on $\mathbb D$ with $\sum_{j=1}^n|b_j(z)|^2 \Le 1,~z \in \mathbb D)$ such that $\log(1-BB^*) \in L^1(\mathbb T).$ Let $A: \mathbb D \rar \mathcal L(\mathbb C^n, \mathbb C^n)$ be the unique outer function such that 
% $B(\zeta)^*B(\zeta) + A(\zeta)^*A(\zeta) = I,$  a.e. $\zeta \in \mathbb T$.
% Let $\varphi = BA^{-1}: \mathbb D \rar \mathcal L(\mathbb C^n, \mathbb C).$ Then $$\mathcal D(T_\varphi^*) =H(B),$$ where $\mathcal D(T)$ denotes the domain of $T.$
\end{theorem}

We follow the approach of \cite[Proposition~5.4]{Sa1} in proving Theorem~\ref{domain-debranges}. We begin with the following lemma.

\begin{lemma} \label{domain-toeplitz}
    Let $B \in \mathcal S(\mathbb C^n, \mathbb C)$ be such that $\log(1-BB^*) \in L^1(\mathbb T).$ Let $A: \mathbb D \rar \mathcal L(\mathbb C^n, \mathbb C^n)$ be an outer function satisfying \eqref{unique-A}, and 
let $\varphi = BA^{-1}: \mathbb D \rar \mathcal L(\mathbb C^n, \mathbb C).$ Then $$\mathcal D(T_\varphi) = A(H^2_{\mathbb C^n}).$$ 
\end{lemma}
\begin{proof}
By \eqref{unique-A}, $A(\zeta)^*A(\zeta) \Le I~\mbox{a.e.}~ \zeta\in \mathbb T$, which implies that each entry of $A$ is in $H^\infty$. Thus, 
$$
Ah \in H^2_{\mathbb C^n},  \quad h \in H^2_{\mathbb C^n}.
$$
Also, since each $b_j \in H^\infty,$ we obtain 
    \beqn
   \varphi Ah  = Bh = \sum_{j=1}^n b_jh_j \in H^2, \quad h=(h_j)_{j=1}^n \in H^2_{\mathbb C^n}.
    \eeqn
    Hence, $Ah \in \mathcal D(T_\varphi).$
    
    To see the reverse inclusion $\mathcal D(T_\varphi) \subseteq A(H^2_{\mathbb C^n}),$ let $h \in \mathcal D(T_\varphi).$ 
   Then,
    \beqn
        (BA^{-1}h)^* (BA^{-1}h) = (A^{-1}h)^*B^*BA^{-1}h 
        &\overset{\eqref{unique-A}}=& (A^{-1}h)^*(I-A^*A)A^{-1}h \\
         &=& (A^{-1}h)^* A^{-1}h - h^*h \quad \text{a.e. on $\mathbb T$}.
    \eeqn
    Since $h \in \mathcal D(T_\varphi)$, it follows that both $(BA^{-1}h)^* (BA^{-1}h)$ and $h^*h$ belong to $L^1(\mathbb T),$ and hence $(A^{-1}h)^*A^{-1}h \in L^1(\mathbb T).$ Equivalently, we have 
    \beq\label{Ainverseh}
    A^{-1}h \in L^2(\mathbb T, \mathbb C^n).
    \eeq
    Moreover, since each entry of $A$ is in $H^\infty$, it follows that each entry of the adjugate $\operatorname{adj(A)}$ of $A$ is in $H^\infty$.  Therefore, $$A^{-1}h= \frac{\operatorname{adj}(A)h}{\det (A)} = \frac{1}{\det(A)}(\tilde{h}_j)_{j=1}^n,$$ where each $\tilde{h}_j$ is in $H^2$.
    % In other words, if \textcolor{red}{we write}
    % \beqn
    % A^{-1}h = \frac{1}{\det(A)}(\tilde{h}_j)_{j=1}^d,
    % \eeqn
    % then it follows that $\frac{\tilde{h}_j}{\det(A)} \in L^2(\mathbb T, \mathbb C)$ for all $j=1,2, \ldots, d.$ 
    By \eqref{Ainverseh}, $\frac{\tilde{h}_j}{\det(A)} \in L^2(\mathbb T)$, and since $A$ is outer, by \cite[pg. 125]{Hel}, $\det(A)$ is outer, and hence by \cite[Corollary~I.3]{Nik}, $\frac{\widetilde{h}_j}{\det(A)} \in H^2$ for all $j=1,2, \ldots, n.$ Therefore, $A^{-1}h \in H^2_{\mathbb C^n}$, which implies that
    \beqn
 h = A(A^{-1}h) \in A(H^2_{\mathbb C^n}).
    \eeqn
    This completes the proof.
    \end{proof}

% We also need the following result that can be easily seen by combining Proposition~5.2 and Theorem~5.1 of \cite{A-M} (see also ). 

 We are now ready to present the proof of Theorem~\ref{domain-debranges}.
\begin{proof}[Proof of Theorem~\ref{domain-debranges}]
To see the inclusion $\mathcal D(T_{\varphi}^*)\subseteq H(B)$, let $f\in \mathcal D(T_{\varphi}^*)$. Then, it follows that $T_A^*T_{\varphi}^*f=T_{B}^*f$, and hence by \eqref{f^+ toeplitz expression}, $f \in H(B)$. 
% proves $\mathcal D(T_{\varphi}^*)\subseteq H(B)$.

Conversely, by Lemma~\ref{domain-toeplitz}, it follows that $T_{\varphi}$ is densely defined and the graph $\mathcal G(T_\varphi)$ of $T_\varphi$ is given by
    \beqn
\mathcal G(T_\varphi) = \{(Ah, Bh)  : h \in H^2_{\mathbb C^n}\} = \mbox{range} \begin{pmatrix}
      T_A \\
      T_B
  \end{pmatrix} .
    \eeqn
This implies that 
\beqn
\mathcal G(T_\varphi)^\perp = \text{ker}(T_A^* \,\,\,\, T_B^*)  = \{(g, f) \in H^2_{\mathbb C^n} \oplus H^2: T_A^*g + T_B^*f = 0\}.
\eeqn
Hence,
\begin{align}
\mathcal G(T_\varphi^*)&= \{(f, g) \in H^2 \oplus H^2_{\mathbb C^n} : (-g) \oplus f \in \mathcal G(T_\varphi)^\perp\}\notag\\
&= \{(f, g) \in H^2 \oplus H^2_{\mathbb C^n} :  T_A^*g=T_B^*f\}.\label{eqngraph}
%&=& \{(f, f^+): f \in H(B)\}.
\end{align}
Now, let $f\in H(B)$. Then, again by \eqref{f^+ toeplitz expression}, there exists $f^+\in H^2_{\mathbb C^n}$ such that $T_{B}^*f+T_A^*f^+=0.$ By \eqref{eqngraph}, $(f,-f^+)\in \mathcal G(T_{\varphi}^*)$, and hence, $f\in \mathcal{D}(T_{\varphi}^*)$ and $T_{\varphi}^*f = -f^+$. Now, by using \eqref{norm-formula}, the remaining norm formula is immediate. This completes the proof. 
\end{proof}

Note that if $B \in \mathcal S(\mathbb C^n, \mathbb C)$ satisfies $\log(1-BB^*) \in L^1(\mathbb T),$ then by using \cite[Lemma 5.3]{A-M}, it follows that $H(B)$ contains all  functions analytic in a neighborhood of $\overline{\mathbb D}.$ The following result provides an explicit norm formula for functions analytic in a neighborhood of $\overline{\mathbb D}.$ This extends the  classical case recorded in \cite[Theorem~4.1]{CGR} (see also, \cite{Sa1}).

\begin{theorem}
\label{norm-formula-thm}
    Let $B \in \mathcal S(\mathbb C^n, \mathbb C)$ be such that $\log(1-BB^*) \in L^1(\mathbb T).$ Let $A: \mathbb D \rar \mathcal L(\mathbb C^n, \mathbb C^n)$ be an outer function given as in \eqref{unique-A}.
Let $BA^{-1}: \mathbb D \rar \mathcal L(\mathbb C^n, \mathbb C)$ admit the series expansion $\sum_{j=0}^\infty c_jz^j$. Suppose  $f$ is a holomorphic function in a neighborhood of $\overline{\mathbb D}$. If $f(z) = \sum_{j=0}^\infty \hat{f}(j) z^j$, then the series $\sum_{j=0}^\infty \hat{f}(j+k){c_j^*}$ converges absolutely for each integer $k \Ge 0$, and 
\beq
\label{norm-formula-nbd}
\|f\|_{H(B)}^2 = \sum_{k=0}^\infty |\hat{f}(k)|^2 + \sum_{k=0}^\infty \Big\|\sum_{j=0}^{\infty} c_j^* \hat{f}(j+k)\Big\|^2.
\eeq
\end{theorem}
\begin{proof}
    We first verify the formula \eqref{norm-formula-nbd} for polynomials. Let $f(z) = \sum_{j=0}^m \hat{f}(j)z^j$. Note that $\varphi(z) =  \sum_{j=0}^m c_j z^j + z^{m+1}\psi_m$ for some  $\mathcal L(\mathbb C^n, \mathbb C)$-valued analytic function $\psi_m$ on $\mathbb D$. By Lemma \ref{domain-toeplitz}, it follows that  $T_{z^{m+1}\psi_m}$ is a densely defined operator in $H^2_{\mathbb C^n}.$
Let $g \in \mathcal D(T_{z^{m+1}\psi_m})$. Then, $z^{m+1}\psi_m g \in H^2$, which implies that $\psi_m g \in H^2$, and therefore,
    \beqn
    \inp{f}{T_{z^{m+1}\psi_m}(g)}_{H^2} = \inp{f}{z^{m+1}\psi_m g}_{H^2} = \inp{T_{z^{m+1}}^*f}{\psi_m g}_{H^2} =  0.
    \eeqn
% Here the last equality follows since \textcolor{red}{for any  $\mathbb C^n$-valued analytic function $h$ on $\mathbb D$, $zh\in H^2_{\mathbb C^n}$ implies $h\in H^2_{\mathbb C^n}$}, and $f$ is a polynomial of degree at most $m$. 
Thus, it follows that $f\in \mathcal D(T_{z^{m+1}\psi_m}^*)$ and $T_{z^{m+1}\psi_m}^*(f) = 0$, and hence
    \beqn
  T_{\varphi}^* (f) = \sum_{j=0}^m T_{c_jz^j}^*(f) = \sum_{j=0}^m c_j^* T_{z^j}^*(f) = \sum_{j=0}^m c_j^* \sum_{k=0}^{m-j} \hat{f}(j+k) z^k = \sum_{k=0}^m\sum_{j=0}^{m-k} c_j^*  \hat{f}(j+k) z^k.
    \eeqn
 Now, by Theorem~\ref{domain-debranges},
    \beq
    \label{formula-poly}
\|f\|_{H(B)}^2 = \|f\|_{H^2}^2 + \|T_{\varphi}^* (f)\|_{H^2_{\mathbb C^n}}^2 = \sum_{k=0}^m |\hat{f}(k)|^2 + \sum_{k=0}^m \Big\|\sum_{j=0}^{m-k} c_j^*  \hat{f}(j+k)\Big\|^2,
\eeq
completing the formula of norm for polynomials. 
% Note that, from \eqref{formula-poly}, we obtain
% \begin{equation}\label{norm of monomials}
%   \|z^m\|_{H(B)}^2=   1 + \sum_{i=0}^{m} \|c_{m-i}^*\|^2=1 + \sum_{i=0}^{m} \|c_{i}^*\|^2,~m\Ge 0.
% \end{equation}

To prove the general case, let $S_m(f)(z) := \sum_{j=0}^m \hat{f}(j) z^j$. Then, by \eqref{formula-poly},
\beq
\label{norm-taylorpoly}
\|S_m(f)\|_{H(B)}^2 = \sum_{k=0}^m |\hat{f}(k)|^2 + \sum_{k=0}^m \Big\|\sum_{j=0}^{m-k} c_j^* \hat{f}(j+k)\Big\|^2, \quad m \Ge 0.
\eeq
By the assumption on $f$, there exists $R > 1$ such that $f$ is holomorphic on the closed disk of radius $R$, and hence, for some constant $d_1>0$, we have 
\beq \label{wj norm}
|\hat{f}(j)| \Le d_1 R^{-j}, \quad j \Ge 0. 
\eeq
Moreover, since $\varphi$ is holomorphic on $\mathbb D$, for any $r \in (1, R)$, there exists $d_2 > 0$ (depending on $r$) such that 
\beq\label{cj norm}
\|c_j\| \Le d_2 r^j, \quad  j \Ge 0.
\eeq
By \eqref{wj norm} and \eqref{cj norm}, we obtain that $\sum_{j=0}^\infty \|c_j^* \hat{f}(j+k)\| \Le \frac{d_1d_2}{R^{k-1}(R-r)}$ for each integer $k \Ge 0$. Hence,
\beq
\label{secod-part-limit}
\lim_{m \rar \infty} \sum_{k=0}^m \Big\|\sum_{j=0}^{m-k} c_j^* \hat{f}(j+k)\Big\|^2 = \sum_{k=0}^\infty \Big\|\sum_{j=0}^{\infty} c_j^* \hat{f}(j+k)\Big\|^2 < \infty.
\eeq
Furthermore, since $1 < r < R$, by \eqref{formula-poly}, \eqref{wj norm} and \eqref{cj norm}, there exists $D>0$ such that 
\beqn
\|\hat{f}(k)z^k\|_{H(B)} = |\hat{f}(k)| \Big(1 + \sum_{i=0}^k \|c_i\|^2 \Big)^{\frac{1}{2}} &\Le& \frac{d_1 \max\{1, d_2\}}{\sqrt{1 - r^2}} \Big(\frac{r}{R}\Big)^k \sqrt{2 r^{-2k} - r^{-2(k-1)} - r^2}\\
&\Le& D \Big(\frac{r}{R}\Big)^k, \quad k \Ge 1. 
\eeqn
This implies that $\{S_m(f)\}_{m \Ge 0}$ is a cauchy sequence in $H(B)$. Also, since norm convergence in $H(B)$ implies pointwise convergence, it follows that $S_m(f)$ converges to $f$ in $H(B)$ as $m \rar \infty$, and hence $\|S_m(f)\|_{H(B)} \rar \|f\|_{H(B)}$ as $m \rar \infty$. Finally, by letting $m \rar \infty$ in \eqref{norm-taylorpoly} and using \eqref{secod-part-limit}, we obtain \eqref{norm-formula-nbd}. This completes the proof. 
\end{proof}

The following result computes the norm of monomials in finite rank $H(B)$-spaces, which is an immediate consequence of \eqref{formula-poly}.

% It is worth noting that we compute the norm of polynomials given in \eqref{formula-poly}. In particular, we obtain the following corollary that computes 
\begin{corollary} 
\label{monomial-formula-lemma-state}
Let  $B \in \mathcal S(\mathbb C^n, \mathbb C)$ be such that $\log(1-BB^*) \in L^1(\mathbb T).$ Suppose $A: \mathbb D \rar \mathcal L(\mathbb C^n, \mathbb C^n)$ is an outer function satisfying \eqref{unique-A}. Let $\varphi = BA^{-1}: \mathbb D \rar \mathcal L(\mathbb C^n, \mathbb C)$ admit the series expansion $\sum_{j=0}^\infty c_jz^j$. 
Then, $z^m \in H(B)$ for all integers $m \Ge 0$, and moreover, 
\beq
\label{monomial-formula-lemma}
\|z^m\|^2_{H(B)} = 1 + \sum_{j=0}^m c_{j} c_{j}^*, \quad m \Ge 0.
\eeq
\end{corollary}

The S\"{z}ego kernel, the reproducing kernel for the scalar valued Hardy space $H^2,$ is given by 
\beq\label{kappa}
\kappa_\lambda(z): = K_{\bf 0}(z,\lambda) = \frac{1}{1 - z\overline{\lambda}}, \quad \lambda, z \in \mathbb D.
\eeq
It is known from \cite[Theorem~5.5]{A-M} that the analytic polynomials are dense in $H(B)$. Since $\kappa_\lambda ~(\lambda \in \mathbb D)$ are contained within $H(B)$-spaces, it is natural to expect a corresponding result regarding denseness. We verify this in the result below, extending the classical case \cite[Corollary~23.26]{FM-2}.
\begin{corollary}\label{kappa-dense}
    Let  $B \in \mathcal S(\mathbb C^n, \mathbb C)$ be such that $\log(1-BB^*) \in L^1(\mathbb T).$ Let $\kappa_\lambda~ (\lambda \in \mathbb D)$ be the function defined as in \eqref{kappa}. Then, the linear span of $\{\kappa_\lambda : \lambda \in \mathbb D\}$ is a dense subspace of $H(B)$.
\end{corollary}
\begin{proof}
For any $h \in H^2_{\mathbb C^n}$, $\lambda \in \mathbb D$ and $v \in \mathbb C^n$, by the reproducing property,
\beqn
\inp{T_A^*(\kappa_\lambda v)}{h}_{H^2_{\mathbb C^n}} = \inp{\kappa_\lambda v}{Ah} = \inp{v}{A(\lambda)h(\lambda)} &=& \inp{A(\lambda)^*v}{h(\lambda)}\\
&=& \inp{\kappa_\lambda (A(\lambda)^*v)}{h}\\
&=& \inp{A(\lambda)^*(\kappa_\lambda v)}{h},
\eeqn
giving that $T_A^*(\kappa_\lambda v)= A(\lambda)^*(\kappa_\lambda v)$. 
Thus, we have 
    \beqn
    T_B^* \kappa_\lambda = B(\lambda)^*\kappa_\lambda = A(\lambda)^*(A(\lambda)^*)^{-1}B(\lambda)^*\kappa_\lambda = T_A^*((BA^{-1})^*(\lambda)\kappa_\lambda),
\eeqn
and hence by \eqref{f^+ toeplitz expression}, $\kappa_\lambda^+ = -(BA^{-1})^*(\lambda)\kappa_\lambda$. Consequently, for any $f \in H(B)$, again by the reproducing property,
    \beqn
\inp{f}{\kappa_\lambda}_{H(B)} = \inp{f}{\kappa_\lambda}_{H^2} + \inp{f^+}{\kappa_\lambda^+}_{H^2_{\mathbb C^n}} &=& f(\lambda) -\inp{f^+}{(BA^{-1})^*(\lambda)\kappa_\lambda}_{H^2_{\mathbb C^n}}\\
 &=& f(\lambda) - (BA^{-1})(\lambda)f^+(\lambda)\\
 &=& f(\lambda) - B(\lambda)A(\lambda)^{-1}f^+(\lambda).
    \eeqn
 Therefore, if $f \perp \text{span}\{\kappa_\lambda  : \lambda \in \mathbb D\}$, then 
    \begin{equation}\label{f+ and f}
            (f - BA^{-1}f^+)(\lambda) = 0, \quad \lambda \in \mathbb D,
    \end{equation}
    and hence $B^*f - B^*BA^{-1}f^+ \equiv 0$ on $\mathbb T$. By \eqref{unique-A}, it now follows that 
    \beq \label{fina-iden}
    B^*f +A^*f^+ = A^{-1}f^+ \quad  \text{on}~~\mathbb T.
    \eeq
   This gives $A^{-1}f^+  \in L^2(\mathbb T, \mathbb C^n).$ 
     Since $A^{-1} = \frac{adj(A)}{det(A)}$ and each entry of $adj(A)$ is in $H^\infty$, by \cite[Corollary~I.3]{Nik}, it follows that each entry of $A^{-1}f^+$ is in $H^2$. This implies that $A^{-1}f^+ \in H^2_{\mathbb C^n}$, which together with \eqref{fina-iden} gives that $B^*f +A^*f^+ \in H^2_{\mathbb C^n}$. Coupling this with \eqref{f^+ toeplitz expression} yields that $A^{-1}f^+ = 0$. Hence, by \eqref{f+ and f}, $f =0$. This completes the proof. 
\end{proof}

\section{Some applications}\label{3}

In this section, we compute the norm of several special functions of $H(B)$, illustrating Theorems~\ref{domain-debranges} \& \ref{norm-formula-thm} from Section~\ref{2}. Let $\{e_i: 1\Le i \Le  n\}$ denote the standard basis of $\mathbb C^n.$
We denote by $L$  the backward shift operator on $H^2$ given by $$(Lf)(z)= \frac{f(z)-f(0)}{z}.$$ 

We begin with an application of Theorem~\ref{domain-debranges}. For the scalar case $n = 1$, the norm of the first formula below is obtained in \cite[pg.~32]{Sa4}, and the last two formulas are obtained in \cite[Corollary~23.9]{FM-2}. 

\begin{corollary}\label{symbolss-norm}
    Let  $B = (b_1,\ldots, b_n)\in \mathcal S(\mathbb C^n, \mathbb C)$ be such that $\log(1-BB^*) \in L^1(\mathbb T).$ Suppose $A: \mathbb D \rar \mathcal L(\mathbb C^n, \mathbb C^n)$ is an outer function satisfying \eqref{unique-A}. 
    % Let $\varphi = BA^{-1}: \mathbb D \rar \mathcal L(\mathbb C^n, \mathbb C)$ admit the series expansion $\sum_{j=0}^\infty c_jz^j$. 
    Let $\kappa_\lambda ~(\lambda \in \mathbb D)$ be the function defined as in \eqref{kappa}. Then, $b_i\kappa_\lambda\in H(B)$ for all $i = 1, \ldots,n$ and $\lambda \in \mathbb D$, and moreover, 
\beq
\label{multiplication-norm}
\|b_i\kappa_\lambda\|_{H(B)}^2 &=& \frac{\|(A(\lambda)^{-1})^*e_i\|^2 -1}{1 - |\lambda|^2}\\
\label{L-combo-norm}   \text{and}~~\,\,\,\, 
\|L(b_i\kappa_\lambda)\|_{H(B)}^2 &=& \frac{\|(A(\lambda)^{-1})^*e_i\|^2 -1}{1 - |\lambda|^2} + 2 \Re\big(\inp{A(\lambda)^{-1}A(0)e_i}{e_i}\big) - |b_i(0)|^2\notag\\ 
&& \,\,\, -\|A(0)e_i\|^2 - \|(A(\lambda)^{-1})^*e_i\|^2, \quad \lambda \in \mathbb D,~~i = 1, \ldots, n,
\eeq
where $\Re(z)$ denotes the real part of $z \in \mathbb C$.
Consequently, $b_i\in H(B)~(i = 1, \ldots,n)$ and
\beqn 
&&\|b_i\|_{H(B)}^2 = \|(A(0)^*)^{-1}e_i\|^2 - 1,\\
&&\|Lb_i\|_{H(B)}^2 = 1 - |B(0)e_i|^2 - \|A(0)e_i\|^2, \quad i=1, \ldots, n.
\eeqn
\end{corollary}
\begin{proof}
Fix $i \in \{1, \ldots, n\}$ and $\lambda \in \mathbb D$. Note first that $\kappa_\lambda e_i \in H^2_{\mathbb C^n}$ is given by 
\beq \label{series-expression}
\kappa_\lambda e_i(z) = \sum_{m=0}^\infty ((\overline{\lambda})^m e_i)z^m, \quad z \in \mathbb D.
\eeq
 As shown in the proof of Corollary~\ref{kappa-dense}, we have $T_A^*(\kappa_\lambda e_i)= A(\lambda)^*(\kappa_\lambda e_i)$.
% \beqn
% \inp{T_A^*(\kappa_\lambda e_i)}{h} = \inp{\kappa_\lambda e_i}{Ah} = \inp{e_i}{A(\lambda)h(\lambda)} &=& \inp{A(\lambda)^*e_i}{h(\lambda)}\\
% &=& \inp{\kappa_\lambda (A(\lambda)^*e_i)}{h}\\
% &=& \inp{A(\lambda)^*(\kappa_\lambda e_i)}{h},
% \eeqn
Since $A(\lambda)$ is invertible, $\kappa_\lambda e_i = (A(\lambda)^*)^{-1}T_A^*(\kappa_\lambda e_i)$, and hence,  
% \beqn
% \label{eigenvector-iden}
% \eeqn
by \eqref{unique-A},    
\beqn
T_B^*(b_i\kappa_\lambda ) = P(B^*B(\kappa_\lambda e_i)) &=& \kappa_\lambda e_i - P((A^*A)(\kappa_\lambda e_i))\\
&=& T_A^* ((A(\lambda)^*)^{-1}\kappa_\lambda e_i) - P(A^*A (\kappa_\lambda e_i))\\
&=&   T_A^*\big(((A(\lambda)^*)^{-1}-A)(\kappa_\lambda e_i)\big).
    \eeqn
Thus, by \eqref{f^+ toeplitz expression},  $b_i\kappa_\lambda  \in H(B)$ with
\beq
\label{membership-comb}
(b_i\kappa_\lambda )^+ = (A-(A(\lambda)^*)^{-1})(\kappa_\lambda e_i).
\eeq
Moreover, by the reproducing property,
\beq
\label{useful-identity}
\inp{A(\kappa_\lambda e_i)}{(A(\lambda)^*)^{-1}\kappa_\lambda e_i}_{H^2_{\mathbb C^n}}^2 = \inp{(A(\lambda))^{-1}A(\kappa_\lambda e_i)}{\kappa_\lambda e_i}_{H^2_{\mathbb C^n}}^2 &=& \inp{(A(\lambda))^{-1}A(\lambda)\kappa_\lambda (\lambda)e_i}{ e_i} \notag\\
&=& \frac{1}{1 - |\lambda|^2}.
\eeq
Now, if we write $\varphi = BA^{-1}: \mathbb D \rar \mathcal L(\mathbb C^n, \mathbb C)$, then by \eqref{membership-comb} and Theorem~\ref{domain-debranges}, $T_\varphi^*(b_i\kappa_\lambda )= ((A(\lambda)^{-1})^* - A)(\kappa_\lambda e_i)$ and therefore,
\beq
\label{combo-norm-in}
\|b_i\kappa_\lambda \|_{H(B)}^2 &=& \|b_i\kappa_\lambda \|_{H^2}^2 + \| T_\varphi^* (b_i\kappa_\lambda )\|_{H^2_{\mathbb C^n}}^2 \notag\\
 &=& \|b_i\kappa_\lambda \|_{H^2}^2 + \| (A-(A(\lambda)^{-1})^*)(\kappa_\lambda e_i)\|_{H^2_{\mathbb C^n}}^2 \notag\\
% &=& \|b_i\kappa_\lambda \|_{H^2}^2 + \|A(\kappa_\lambda e_i)\|_{H^2_{\mathbb C^n}}^2 + \|(A(\lambda)^{-1})^*(\kappa_\lambda e_i)\|_{H^2_{\mathbb C^n}}^2  - 2 \Re\big(\inp{A(\kappa_\lambda e_i)}{(A(\lambda)^{-1})^*(\kappa_\lambda e_i)}\big) \notag\\
&\overset{\eqref{useful-identity}}=& \|B(\kappa_\lambda e_i)\|_{H^2}^2 + \|A(\kappa_\lambda e_i)\|_{H^2_{\mathbb C^n}}^2  + \|(A(\lambda)^{-1})^*(\kappa_\lambda e_i)\|_{H^2_{\mathbb C^n}}^2  - \frac{2}{1-|\lambda|^2}\notag \\
&\overset{\eqref{unique-A}}=&  \|\kappa_\lambda e_i\|_{H^2_{\mathbb C^n}}^2+ \|(A(\lambda)^{-1})^*(\kappa_\lambda e_i)\|_{H^2_{\mathbb C^n}}^2  - \frac{2}{1-|\lambda|^2} \notag\\
&\overset{\eqref{series-expression}}=& \frac{\|(A(\lambda)^{-1})^*e_i\|^2 -1}{1 - |\lambda|^2}.
\eeq
This proves the formula \eqref{multiplication-norm}.

To see \eqref{L-combo-norm}, since $L$ is a bounded operator on $H(B)$, it follows that $L(b_i\kappa_\lambda ) \in H(B)$. Since $(Lb_i\kappa_\lambda )^+ = L ((b_i\kappa_\lambda )^+)$ (see, \cite[Theorem~2.2(ii)]{CGL}), we get 
\beqn
(Lb_i\kappa_\lambda )^+(z)  &\overset{\eqref{membership-comb}}=& L(A(\kappa_\lambda e_i)) - (A(\lambda)^{-1})^*L(\kappa_\lambda e_i), 
\eeqn
and hence, by Theorem~\ref{domain-debranges}, $T_\varphi^*(L(b_i\kappa_\lambda )) = L\big((A(\lambda)^{-1})^*(\kappa_\lambda e_i) - (A(\kappa_\lambda e_i))\big)$. Therefore, we get
\beqn
\|L(b_i\kappa_\lambda )\|_{H(B)}^2 &=& \|L(b_i\kappa_\lambda )\|_{H^2}^2 + \|L\big((A(\kappa_\lambda e_i)) - (A(\lambda)^{-1})^*(\kappa_\lambda e_i)\big)\|_{H^2_{\mathbb C^n}}^2\\
&=& \|b_i\kappa_\lambda \|_{H^2}^2  \! + \|(A- (A(\lambda)^{-1})^*) \kappa_\lambda e_i\|_{H^2_{\mathbb C^n}}^2 \!\!\!- |b_i(0)|^2 \!- \! \|(A(0)- (A(\lambda)^{-1})^*)e_i\|^2\\
&=& \|b_i\kappa_\lambda \|_{H(B)}^2 - |b_i(0)|^2 - \|(A(0)- (A(\lambda)^{-1})^*)e_i\|^2\\
&\overset{\eqref{combo-norm-in}}=& \frac{\|(A(\lambda)^{-1})^*e_i\|^2 -1}{1 - |\lambda|^2} - |b_i(0)|^2 -  \|A(0)e_i- (A(\lambda)^{-1})^*e_i\|^2  \\
&=& \frac{\|(A(\lambda)^{-1})^*e_i\|^2 -1}{1 - |\lambda|^2} - |b_i(0)|^2 - \|A(0)e_i\|^2 - \|(A(\lambda)^{-1})^*e_i\|^2 \\
&&+ 2 \Re\big(\inp{A(\lambda)^{-1}A(0)e_i}{e_i}\big).
\eeqn
This proves \eqref{L-combo-norm}. Finally, since $k_0 \equiv 1$, letting $\lambda = 0$ in \eqref{multiplication-norm} and \eqref{L-combo-norm} yields the remaining formulas. This concludes the proof.
\end{proof}

 The following is an application of Theorem~\ref{norm-formula-thm}. For $n=1$, norm of functions below is computed in \cite[pg.~81]{Sa2} and \cite[pg.~32]{Sa4}.

\begin{corollary}\label{monomial-norm}
    Let  $B \in \mathcal S(\mathbb C^n, \mathbb C)$ be such that $\log(1-BB^*) \in L^1(\mathbb T).$ Suppose $A: \mathbb D \rar \mathcal L(\mathbb C^n, \mathbb C^n)$ is an outer function satisfying \eqref{unique-A}. Let $\varphi = BA^{-1}: \mathbb D \rar \mathcal L(\mathbb C^n, \mathbb C)$ admit the series expansion $\sum_{j=0}^\infty c_jz^j$. If $\kappa_\lambda $ $(\lambda \in \mathbb D)$ is the function defined as in \eqref{kappa},
then $z^m\kappa_\lambda \in H(B)$ for all integers $m \Ge 0$ and $\lambda \in \mathbb D$, and moreover, for all $m \Ge 1$,
\beq
\label{monomial-combi-norm}
\|z^m\kappa_\lambda \|_{H(B)}^2 &=& \frac{1+B(\lambda)(A(\lambda)^*A(\lambda))^{-1}B(\lambda)^*}{1-|\lambda|^2} + \sum_{k=0}^{m-1}\Big\|\sum_{j=0}^{\infty} c_{j+m-k}^* (\overline{\lambda})^{j}\Big\|^2 \\
&\text{and}&\label{kernelfunction-formula}\,\|\kappa_\lambda \|_{H(B)}^2 = \frac{1+B(\lambda)(A(\lambda)^*A(\lambda))^{-1}B(\lambda)^*}{1-|\lambda|^2}, \quad \lambda \in \mathbb D.
\eeq
% Consequently,
% \beq
% \label{monomial-formula}
% \|z^m\|^2_{H(B)} = 1 + \sum_{j=0}^m c_{j} c_{j}^*, \quad m \Ge 0.
% \eeq
\end{corollary}
\begin{proof} With notations as in Theorem~\ref{norm-formula-thm}, $f = z^m\kappa_\lambda = \sum_{j = 0}^\infty \hat{f}(j) z^j,~z \in \mathbb D$, where $\hat{f}(j) = (\overline{\lambda})^{j-m}$ if $j \Ge m$, and $0$, otherwise. If $m \Ge 1$, then it follows from Theorem~\ref{norm-formula-thm} that
\beqn
\|z^m\kappa_\lambda \|_{H(B)}^2  &=& \sum_{j=0}^\infty |\hat{f}(j)|^2 + \sum_{k=0}^\infty \Big\|\sum_{j=0}^{\infty} c_j^* \hat{f}(j+k)\Big\|^2\\
&=&  \sum_{j=m}^\infty |\lambda|^{2(j-m)}  + \sum_{k=m}^\infty \Big\|\sum_{j=0}^{\infty} c_j^* (\overline{\lambda})^{j+k-m}\Big\|^2 +\sum_{k=0}^{m-1} \Big\|\sum_{j=m-k}^{\infty} c_j^* (\overline{\lambda})^{j+k-m}\Big\|^2\\
&=& \sum_{j=m}^\infty |\lambda|^{2(j-m)} +  \sum_{k=m}^\infty |\lambda|^{2(k-m)} \Big\|\sum_{j=0}^{\infty} c_j^* (\overline{\lambda})^{j}\Big\|^2 + \sum_{k=0}^{m-1} \Big\|\sum_{j=0}^{\infty} c_{j+m-k}^* (\overline{\lambda})^{j}\Big\|^2\\
&=& \frac{1}{1-|\lambda|^2} \Big(1 + \Big\|\sum_{j=0}^{\infty} c_j^* (\overline{\lambda})^{j}\Big\|^2\Big) + \sum_{k=0}^{m-1} \Big\|\sum_{j=0}^{\infty} c_{j+m-k}^* (\overline{\lambda})^{j}\Big\|^2\\
&=& \frac{1 + \|(BA^{-1})^*(\lambda)\|^2}{1-|\lambda|^2} + \sum_{k=0}^{m-1} \Big\|\sum_{j=0}^{\infty} c_{j+m-k}^* (\overline{\lambda})^{j}\Big\|^2.
\eeqn
This proves \eqref{monomial-combi-norm}. To verify \eqref{kernelfunction-formula}, we repeat the steps used to derive \eqref{monomial-combi-norm} by neglecting the final term in each expression and the result follows immediately.
% Now, the part (i) is immediate.  To see (ii), we let $m = 0$ in \eqref{monomial-combi-norm} and obtain 
%     \beqn
% \|\kappa_\lambda \|_{H(B)}^2 
% = \frac{1}{1-|\lambda|^2}\Big(1 + \Big\|\sum_{j=0}^{\infty} c_j^* (\overline{\lambda})^{j}\Big\|^2\Big)
% &=&  \frac{1 + \|(BA^{-1})^*(\lambda)\|^2}{1-|\lambda|^2}\\
% &=&  \frac{1+B(\lambda)(A(\lambda)^*A(\lambda))^{-1}B(\lambda)^*}{1-|\lambda|^2}, \quad \lambda \in \mathbb D.
%     \eeqn
    This finishes the proof.
% (i): Since $\log(1-BB^*) \in L^1(\mathbb T)$, by \cite[Lemma~2.5]{CGL}, $z^m \in H(B)$ for every $m \Ge 0$. Now, by Theorem~\ref{norm-formula-thm}, we obtain 
\end{proof}

% \begin{example}[Example~\ref{explicit-ex1} continued...]

% \begin{corollary}\label{symbols-norm}
%     Let  $B = (b_1,\ldots, b_n)\in \mathcal S(\mathbb C^n, \mathbb C)$ such that $\log(1-BB^*) \in L^1(\mathbb T).$ Suppose $A: \mathbb D \rar \mathcal L(\mathbb C^n, \mathbb C^n)$ be an outer function satisfying \eqref{unique-A}. Then, 
% \end{corollary}
% \begin{proof}
	
% \end{proof}

Below, we present an example of $H(B)$ and provide explicit computations of norm formulas for various functions of the space using the above results. 

\begin{example}\label{explicit-ex1}
    Let $u$ be a non-constant inner function and let $\omega \in \mathbb C$ be such that $\omega^3 = 1$, $w\neq 1$. Consider $B = (b_1, b_2)$ with $b_j = \frac{1+ \omega^{j-1}u}{\sqrt{6}},~j=1,2$. 
    Since $u$ is non-constant and inner, $1+\omega^2 u$ is analytic and $\Re(1+\omega^2 u) > 0$ on unit disc, and hence by \cite[pg.~51, Exercise~1]{Duren}, $1+\omega^2 u$ is outer. Since $\sum_{j=0}^2 |1 + \omega^{j}u|^2 = 6$ a.e. on $\mathbb T$,  a scalar valued mate $a$ is given by
    $$
    a =  \frac{1+ \omega^{2}u}{\sqrt{6}}, \quad z \in \mathbb D.
    $$
     Moreover, by using the identities $1 + \omega + \omega^2 = 0$ and $\omega^2 = \overline{\omega}$, we have
    $$
    B^*B = \frac{1}{6}
    \begin{pmatrix}
        2 + u +  \overline{u}  & \omega u + \overline{u} - \omega^2 \\
        u + \omega^2 \overline{u} - \omega  & 2 + \omega u + \overline{\omega u}
    \end{pmatrix} \quad \text{a.e. on $\mathbb T$},
    $$
    so that 
    $$
    I - B^*B = \frac{1}{6}
    \begin{pmatrix}
        4 -u  - \overline{u}  & \omega^2- \omega u - \overline{u}   \\
       \omega - u - \omega^2 \overline{u}   & 4 - \omega u -\overline{\omega u}
    \end{pmatrix} \quad \text{a.e. on $\mathbb T$}.
    $$
   If we let 
   $$
   A = \frac{1}{\sqrt{6}} \begin{pmatrix}
        1 - u  & 1 - \omega u \\
       \sqrt{2}  & \sqrt{2} \omega^2
    \end{pmatrix},
   $$
   then it follows after some steps of calculations that
   $$
   A^*A = \frac{1}{6} \begin{pmatrix}
        |1 - u|^2 + 2  & (1 - \overline{u}) (1 - \omega u) + 2\omega^2  \\
        (1 - \overline{\omega u}) (1 - u) + 2 \omega & |1 - \omega u|^2 +2
    \end{pmatrix}
   =  I - B^*B \quad \text{a.e. on $\mathbb T$}.
   $$
   To verify that $A$ is outer, note first that
$$
\det(A) = \frac{\sqrt{2}}{6} \omega^2 (1 - \omega) (1 + \omega^2 u).
$$
Since $u$ is non-constant, $1+\omega^2 u$ is analytic and $\Re(1+\omega^2 u) > 0$ on unit disc, and hence by \cite[pg.~51, Exercise~1]{Duren}, $1+\omega^2 u$ is outer. Therefore, $\det(A)$ is outer, and so by \cite[pg. 125]{Hel}, $A$ is outer. In this case, $B(0) \neq 0$, and 
$$
(A(0)^*)^{-1} = \frac{\sqrt{3}}{(\omega - 1) (1 + \omega \overline{u(0)})}
\begin{pmatrix}
        \sqrt{2} \omega  & -\sqrt{2} \\
        \omega^2 \overline{u(0)} - 1 & 1 - \overline{u(0)}
    \end{pmatrix}.
$$ Thus, by Corollary~\ref{symbolss-norm}, $b_i~ (i=1,2)$ belongs to $H(B)$, and
$$
\|b_1\|_{H(B)} = \frac{ \sqrt{2(1 + \Re(u(0))}} {|1+\omega \overline{u(0)}|}, 
$$
and 
$$
\|b_2\|_{H(B)} = \frac{\sqrt{2(1 + \Re(\omega \,u(0))}} {|1+\omega \overline{u(0)}|}. 
$$
Furthermore, we obtain 
$$
BA^{-1} = 
 \begin{pmatrix}
        \frac{\omega - u}{\omega + u}  & \frac{-\sqrt{2} \omega^2 u}{\omega + u}
    \end{pmatrix}.
$$
% This yields that BA^{-1}(BA^{-1})^* = \frac{|\omega - }
For instance, if $u(z) = z$, then it can be easily seen that
$$
BA^{-1}(z) = (1, 0) + \sum_{m=1}^\infty 
\begin{pmatrix}
        2(-\omega^{2})^m  & \sqrt{2}(-\omega^2)^m \omega^{2}
    \end{pmatrix} 
    z^m, \quad z \in \mathbb D.
$$
If $\{c_m\}_{m=0}^\infty$ denotes the sequence of Taylor coefficients of $BA^{-1}$, then $c_0 c_0^* = 1$ and 
$$
c_m c_m^* = 6, \quad m \Ge 1.
$$
Therefore, by \eqref{monomial-formula-lemma}, $\|z^m\|^2_{H(B)} = 2 + 6m, ~ m \Ge 0$.
% & \|\kappa_\lambda \|^2_{H(B)} = (1 - |\lambda|^2)^{-1} \Big(2 - 4 \Re\big(\frac{\omega^2 \lambda}{1 + \omega^2 \lambda}\big) + 6\sum\limits_{m,n=1}^\infty (-1)^{m+n}\omega^{2m+n} \lambda^m \overline{\lambda}^n \Big), \quad \lambda \in \mathbb D.
\end{example}

 % \section{An explicit norm formula of elements of $H(B)$}\label{3}

% Let $\Delta \in L^\infty(\mathbb T, M_n(\mathbb C))$ be defined by
% \beqn
% \Delta(w) = (I- B(w)^*B(w))^{1/2}, \quad \mbox{a.e.\,}\, w \in \mathbb T.
% % = 
% % \begin{pmatrix}
% % 1 - |b_1(w)|^2 & -\overline{b_1(w)}b_2(w)\\
% % -\overline{b_2(w)}b_1(w) & 	1 - |b_2(w)|^2
% % \end{pmatrix}^{1/2}.
% \eeqn
% Then, for any $f \in H^2(\mathbb T, \mathbb C^n),$ define
% \beqn
% (\Delta f)(w) = \Delta(w)f(w), \quad \mbox{a.e.\,}\, w \in \matthbb T.
% \eeqn 
% Let $\mathcal K$ and $U$ be the subspaces of $ H^2(\mathbb T, \mathbb C) \oplus L^2(\mathbb T, \mathbb C^n) $ defined by 
% \beqn
% \mathcal K: = H^2(\mathbb T, \mathbb C) \oplus \overline{\Delta H^2(\mathbb T, \mathbb C^n)},\,\,U: = \{(Bf, \Delta f): f \in H^2(\mathbb T, \mathbb C^n)\}.
% \eeqn
% Then, the projection $P: \mathcal K \rar H^2(\mathbb T, \mathbb C)$ defined by $P(f \oplus g) = f,$ is an injective function on $\mathcal K \ominus U.$ This implies that for every $f \in H(B),$ there exists a unique element $f^+ \in \overline{\Delta H^2(\mathbb T, \mathbb C^n)}$ such that $(f, f^+) \in \mathcal K \ominus U,$ and 
% \beqn
% \|f\|_{H(B)}^2 = \|f\|_{H^2(\mathbb T, \mathbb C)}^2 + \|f^+\|_{L^2(\mathbb T, \mathbb C^n)}^2, \quad f \in H(B),
% \eeqn
% see \cite{}. 

% An immediate consequence of Theorem~\ref{norm-formula-thm} yields the norm of Szeg$\ddot{o}$ kernel at points of $\mathbb D$. 

% \begin{corollary}\label{kernel-norm}
%  With the notations as in Theorem~\ref{norm-formula-thm}, 
% \end{corollary}
% \begin{proof}
  
% \end{proof}

\section{On the inclusion $H^\infty \subseteq H(B)$}\label{4}

In this section,  our main goal is to characterize when finite rank $H(B)$ contains $H^\infty$.

% our main goal is to prove Theorem C, which we restate below for convenience.

% This extends the classical case proved by Sarason \cite[Theorem~1]{Sa2}.
\begin{theorem}\label{bounded-contained}
    Let  $B \in \mathcal S(\mathbb C^n, \mathbb C)$ be such that $\log(1-BB^*) \in L^1(\mathbb T).$ Let $a$ and $A$ be outer functions satisfying \eqref{scalar outer a} and \eqref{unique-A}, respectively. Then, the following are equivalent:
    \begin{itemize}
        \item [(i)] $H^\infty \subseteq H(B)$,
        \item [(ii)] $\sup_{m \Ge 0} \|z^m\|_{H(B)} < \infty$,
        \item [(iii)] $BA^{-1} \in H^2_{\mathcal L(\mathbb C^n, \mathbb C)}$,
        \item[(iv)] $\frac{B}{a}\in H^2_{\mathcal L(\mathbb C^n, \mathbb C)}$,
        \item [(v)] $(1 - BB^*)^{-1} \in L^{1}(\mathbb T)$.
        \end{itemize}
    \end{theorem}
    
We need a couple of lemmas to prove Theorem~\ref{bounded-contained}. Recall that the {\it Smirnov class}, denoted by $N^+(\mathcal L(\mathbb C^n, \mathbb C))$, consists of analytic functions of the form $\frac{u}{v}$, where $u\in H^\infty_{\mathcal L(\mathbb C^n, \mathbb C)}$ and $v\in H^{\infty}$ is outer, see \cite[Section~4.7]{Ros-Rov}. 
% \iff \int_{\mathbb T} \|f(\zeta)\|^2 dm(\zeta) < \infty
\begin{lemma}
\label{smirnov}
  Let $B \in \mathcal S(\mathbb C^n, \mathbb C)$ be such that $\log(1-BB^*) \in L^1(\mathbb T).$ Let $a$ and $A$ be outer functions satisfying \eqref{scalar outer a} and \eqref{unique-A}, respectively. Then, both $\frac{B}{a}$ and $BA^{-1}$ belong to $N^+(\mathcal L(\mathbb C^n, \mathbb C))$.
      % \item [(ii)] for any $h \in H^2_{\mathbb C^n}$, $BA^{-1}h \in N^+(\mathbb C)$.
  \end{lemma}
\begin{proof} Since $B \in \mathcal S(\mathbb C^n, \mathbb C)$, by definition, $\frac{B}{a} \in N^+(\mathcal L(\mathbb C^n, \mathbb C))$. 
For the other function, note first that since $A$ is outer, by \cite[pg. 125]{Hel}, $\det(A)$ is outer.
Moreover, 
$A^{-1}$ can be written as $\frac{\operatorname{adj}(A)}{\det (A)}$. Since $A(\zeta)^*A(\zeta) \Le I~\mbox{a.e.}~ \zeta\in \mathbb T$, it follows that each entry of $\operatorname{adj(A)}$ is in $H^\infty$. Also, since each component of $B$ is in $H^\infty$,  each component of $B\operatorname{adj}(A)$ is in $H^\infty$, and hence 
$$
BA^{-1} = \frac{B\operatorname{adj}(A)}{\det (A)} \in N^+(\mathcal L(\mathbb C^n, \mathbb C)).
$$
This concludes the proof.
% For the other function, we use  note first that since $A$ is outer, by \cite[pg. 125]{Hel}, $\det(A)$ is outer.
% Moreover, 
% $A^{-1}$ can be written as $\frac{\operatorname{adj}(A)}{\det (A)}$, where $\operatorname{adj}(A):\mathbb D\to \mathcal L(\mathbb C^n, \mathbb C^n)$ is the adjugate of $A$. Since $A(\zeta)^*A(\zeta) \Le I~\mbox{a.e.}~ \zeta\in \mathbb T$, it follows that each entry of $\operatorname{adj(A)}$ is in $H^\infty$. Also, since each component of $B$ is in $H^\infty$,  each component of $B\operatorname{adj}(A)$ is in $H^\infty$. Hence, 
% $$
% BA^{-1} = \frac{B\operatorname{adj}(A)}{\det (A)} \in N^+(\mathcal L(\mathbb C^n, \mathbb C)).
% $$
% % Therefore, $A^{-1}h=\frac{1}{\det(A)}(\tilde{h}_j)_{j=1}^n$, where $\tilde{h}_j\in H^2$, $j=1,\ldots,n.$
% %  Clearly, . Further, it follows  from the proof of Lemma \ref{domain-toeplitz} that  $A^{-1}=\frac{\operatorname{adj}(A)}{\det (A)}$, where each entry of $\operatorname{adj(A)}$ is in $H^\infty$ and $\det(A)$ is outer.
% % Therefore, it follows that  
\end{proof}
We also need the Smirnov maximum principle (see, \cite[Section~4.7, Theorem~A(i)]{Ros-Rov}), which asserts that if $f \in N^+(\mathcal L(\mathbb C^n, \mathbb C))$, then
\beq\label{smirnov-max}
\text{$f \in L^p(\mathbb T, \mathcal L(\mathbb C^n, \mathbb C))   \iff f \in H^p_{\mathcal L(\mathbb C^n, \mathbb C)}$ for any $p \in (0, \infty]$}.
\eeq

The following lemma establishes a relation between the boundary values of both outer functions $a$ and $A$. 
\begin{lemma}
% \label{a&A relation}
    Let $B \in \mathcal S(\mathbb C^n, \mathbb C)$ be such that $\log(1-BB^*) \in L^1(\mathbb T).$ Let $a$ and $A$ be outer functions satisfying \eqref{scalar outer a} and \eqref{unique-A}, respectively. Then, the following identity holds. 
        \beq \label{a-andA relation}
B(\zeta)(A(\zeta)^*A(\zeta))^{-1}B(\zeta)^* = \frac{B(\zeta)B(\zeta)^*}{|a(\zeta)|^2} \quad \mbox{a.e.\,}\, \zeta \in \mathbb T.
\eeq
% Moreover, 
%     \begin{itemize}
%         \item[(i)]  $BA^{-1}\in H^2_{\mathcal L(\mathbb C^n, \mathbb C)}$ if and only if $\frac{B}{a}\in H^2_{\mathcal L(\mathbb C^n, \mathbb C)}$,
% \item[(ii)]  $BA^{-1}\in H^\infty_{\mathcal L(\mathbb C^n, \mathbb C)}$ if and only if $\frac{B}{a}\in H^\infty_{\mathcal L(\mathbb C^n, \mathbb C)}$.
%     \end{itemize}
\end{lemma}
\begin{proof}
 For a.e. $\zeta \in \mathbb T$, by \eqref{scalar outer a} and \eqref{unique-A}, 
\beqn
B(\zeta)(A(\zeta)^*A(\zeta))^{-1}B(\zeta)^* &=& B(\zeta)(I- B(\zeta)^*B(\zeta))^{-1}B(\zeta)^*\\ 
&=& B(\zeta)\Big(I+ B(\zeta)^*(1 - B(\zeta)B(\zeta)^*)^{-1}B(\zeta)\Big)B(\zeta)^*\\
&=&B(\zeta)B(\zeta)^* \big(1 + (1 - B(\zeta)B(\zeta)^*)^{-1} B(\zeta)B(\zeta)^*\big)\\
&=& B(\zeta)B(\zeta)^* \Big(1 + \frac{B(\zeta)B(\zeta)^*}{|a|^2}\Big)\\
&=& \frac{B(\zeta)B(\zeta)^*}{|a(\zeta)|^2},
\eeqn
proving \eqref{a-andA relation}. 
% (i): By Lemma~\ref{smirnov}, $BA^{-1} \in  N^+(\mathcal L(\mathbb C^n, \mathbb C))$, which combined with the Smirnov maximum principle (see, \eqref{smirnov-max}) implies that  $BA^{-1}\in H^2_{\mathcal L(\mathbb C^n, \mathbb C)}$ if and only if $BA^{-1}(BA^{-1})^* \in L^1(\mathbb T).$ By \eqref{a&A relation}, this is further equivalent to $\frac{BB^*}{|a|^2}\in L^1(\mathbb T).$ By Lemma~\ref{smirnov}, $\frac{B}{a} \in N^+(\mathcal L(\mathbb C^n, \mathbb C))$, so that by \eqref{smirnov-max}, the equivalence of part (i) follows.
% (ii)  
% This finishes the proof.
% Moreover, since $H^\infty_{\mathcal L(\mathbb C^n, \mathbb C)} \subseteq H^2_{\mathcal L(\mathbb C^n, \mathbb C)}$, by the part (i) above, it follows that $\frac{B}{a} \in H^2_{\mathcal L(\mathbb C^n, \mathbb C)}$. This together with \eqref{B and scalar a} concludes that $\frac{B}{a} \in H^\infty_{\mathcal L(\mathbb C^n, \mathbb C)}$. The converse part can be proved exactly in the same way as the forward part. 
% Then, $BA^{-1}\in $, which combined with the part (i) implies that . 
% \textcolor{red}{ The proof of part (ii) is similar} to that of part (i) and is thus omitted.
\end{proof}

\begin{proof}[Proof of Theorem~\ref{bounded-contained}]
   (i) $\implies$ (ii) By an application of the closed graph theorem, this implication is immediate.
   
   (ii) $\iff$ (iii) Let $BA^{-1}: \mathbb D \rar \mathcal L(\mathbb C^n, \mathbb C)$ admit the series expansion $\sum_{j=0}^\infty c_jz^j$. Then, by Corollary~\ref{monomial-formula-lemma-state}, we have
   $$\sup_{m \Ge 0} \|z^m\|^2_{H(B)} = 1+\sum_{j=0}^\infty \|c_j\|^2.$$ 
   This concludes that $\sup_{m \Ge 0} \|z^m\|_{H(B)} < \infty$  if and only if $BA^{-1} \in  H^2_{\mathcal L(\mathbb C^n, \mathbb C)}$.

(iii) $\iff$ (iv) By \eqref{a-andA relation}, observe  that $BA^{-1}\in L^2_{\mathcal L(\mathbb C^n, \mathbb C)}$ if and only if $\frac{B}{a}\in L^2_{\mathcal L(\mathbb C^n, \mathbb C)}$. Since $BA^{-1}, ~\frac{B}{a} \in N^+(\mathcal L(\mathbb C^n, \mathbb C))$ (see, Lemma~\ref{smirnov}), by \eqref{smirnov-max}, the desired equivalence is immediate. 

% By Lemma~\ref{a&A relation}(i), this equivalence  follows.
% $(iii) \implies (iv)$: Suppose that $BA^{-1} \in H^2_{\mathbb C^n}$. Then, $B(A^*A)^{-1}B^* = BA^{-1}(BA^{-1})^* \in L^1(\mathbb T)$. This together with Lemma \ref{a&A relation} yields $\frac{BB^*}{|a|^2}\in L^1(\mathbb T).$ Since each $\frac{b_j}{a}$, $j=1,\ldots,n$, is in the Smirnov class $N^+$, by Smirnov maximum principle it follows that $\frac{B}{a}\in H^2_{\mathbb C^n}.$
   
   (iv) $\iff$ (v) Note that for a.e. $\zeta \in \mathbb T$,
   \beq
   \label{an-imp-relation-local}
(1 - BB^*)^{-1}(\zeta) = \frac{1}{|a(\zeta)|^2} \overset{\eqref{scalar outer a}}= 1+\frac{B(\zeta)B(\zeta)^*}{|a(\zeta)|^2}.
   \eeq
  Since  $\frac{B}{a} \in \mathcal N^+(\mathcal L(\mathbb C^n, \mathbb C))$ (see, Lemma~\ref{smirnov}), it follows from \eqref{smirnov-max} that
   $\frac{B}{a} \in H^2_{\mathcal L(\mathbb C^n, \mathbb C)}$ if and only if $\frac{B(\zeta)B(\zeta)^*}{|a(\zeta)|^2}\in L^1(\mathbb T)$. Coupling this with \eqref{an-imp-relation-local} proves the required equivalence.  

   (v) $\implies$ (i) Assume that $(1 - BB^*)^{-1} \in L^{1}(\mathbb T)$. Since $(1 - BB^*)^{-1} = |a|^{-2}$ and $a$ is outer, by \eqref{smirnov-max}, $\frac{1}{a} \in H^2$. Let $h \in H^\infty$. Then, $\frac{h}{a} \in H^2$, which implies that
    \beqn
   h = a \times \frac{h}{a} \in \mathcal M(a).
    \eeqn
    Since $\mathcal M(a) \subseteq \mathcal M(\overline{a}) \subseteq H(B)$ (see, \cite[Lemma~2.5]{CGL}), we obtain that $h \in H(B)$. This completes the proof. 
\end{proof}

 The following result describes when a finite rank $H(B)$-space coincides with $H^2$ with equivalence of norms. 
 
\begin{proposition}
    Let  $B\in \mathcal S(\mathbb C^n, \mathbb C)$ be such that $\log(1-BB^*) \in L^1(\mathbb T).$ Let $a$ and $A$ be outer functions satisfying \eqref{scalar outer a} and \eqref{unique-A}, respectively. Then the following are equivalent:
    \begin{itemize}
        \item[(i)] $H(B)=H^2 (\mbox{with equivalence of norms}),$
        \item[(ii)] $\sup_{z\in \mathbb D}\|B(z)\|_{_{\mathcal L(\mathbb C^n, \mathbb C)}} < 1,$ \vspace{0.8mm}
        \item[(iii)] $BA^{-1}\in H^\infty_{\mathcal L(\mathbb C^n, \mathbb C)}$,
        \item[(iv)] $\frac{B}{a}\in H^\infty_{\mathcal L(\mathbb C^n, \mathbb C)}.$ 
    \end{itemize}
\end{proposition}
\begin{proof}
(i) $\iff$ (ii) This can be proved in the same way as in the classical case; refer to \cite[Corollary 18.6]{FM-2}.

(ii) $\implies$ (iii) Suppose (ii) holds. Then, by \eqref{scalar outer a}, there exists $c > 0$ such that $|a(\zeta)|\Ge c$ a.e. $\zeta\in\mathbb T$. By Lemma \ref{a-andA relation}, it now follows that $BA^{-1}\in L^{\infty}(\mathbb C^n, \mathbb C)$. Since $BA^{-1} \in \mathcal N^+(\mathcal L(\mathbb C^n, \mathbb C))$ (see, Lemma~\ref{smirnov}), it follows from \eqref{smirnov-max} that $BA^{-1}\in H^\infty_{\mathcal L(\mathbb C^n, \mathbb C)}$.

(iii) $\iff$ (iv) By \eqref{a-andA relation}, it follows that $BA^{-1}\in L^\infty_{\mathcal L(\mathbb C^n, \mathbb C)}$ if and only if $\frac{B}{a}\in L^\infty_{\mathcal L(\mathbb C^n, \mathbb C)}$. Since $BA^{-1}, ~\frac{B}{a} \in N^+(\mathcal L(\mathbb C^n, \mathbb C))$, by \eqref{smirnov-max}, the desired equivalence follows.

(iv) $\implies$ (ii) Assume that (iv) holds. Then, there exists $c>0$ such that $\frac{B(\zeta)B(\zeta)^*}{|a(\zeta)|^2} \Le c$ a.e. $\zeta\in\mathbb T$. This together with the relation 
$$
(1 - BB^*)^{-1} = 1+\frac{BB^*}{|a|^2} \quad \text{a.e. on $\mathbb T$}
$$
 concludes that $$B(\zeta)B(\zeta)^*\Le \frac{c}{1+c}<1 \quad \mbox{a.e.}~ \zeta\in\mathbb T.$$ Hence, $\sup_{z\in\mathbb D}\|B(z)\|_{_{\mathcal L(\mathbb C^n, \mathbb C)}}<1.$ This completes the proof.
\end{proof}

\textbf{Concluding remarks.}
Under the Szeg\H{o}-type condition \( \log(1 - BB^*) \in L^1(\mathbb T) \), the space \( H(B) \) admits a concrete realization as the domain of the adjoint of a Toeplitz operator with symbol \( \varphi = BA^{-1} \), where \( A \) is the associated outer function. In this formulation, the norm on \( H(B) \) coincides with the graph norm of \( T_\varphi^* \), and the action of \( T_\varphi^* \) admits an explicit coefficient description as a discrete convolution operator,
\[
k \mapsto \sum_{j = 0}^\infty c_j^* \hat f(j+k).
\]
This yields a concrete sequence-space model for \( H(B) \), making both the operator and norm structure fully explicit. In the classical $H(b)$ spaces (rank-one) setting, related norm representations are known (see, \cite{CGR}), and the present formulation extends such descriptions to the finite rank vector-valued case and identifies explicitly the associated operator acting on coefficients.

The characterization of the inclusion \( H^\infty \subseteq H(B) \) further identifies the precise threshold for the size of the space, showing that it is equivalent to the conditions \( BA^{-1} \in H^2_{\mathcal L(\mathbb C^n, \mathbb C)} \) and \( (1 - BB^*)^{-1} \in L^1(\mathbb T) \).  It should be emphasized, however, that the arguments rely essentially on the  \emph{finite rank hypothesis} $B \in \mathcal S(\mathbb C^n, \mathbb C).$ In the absence of this finite dimensional structure, substantial obstructions arise, notably in the existence of inverse  of an operator-valued outer function. Extending the present framework to infinite dimensional settings therefore remains a challenging problem, likely requiring new methods beyond those employed here, and constitutes a natural direction for further investigation.

% \begin{proposition}
%     Let $\varphi\in N^+( \mathcal L(\mathbb C^n, \mathbb C))$. Then there exists a $B\in S(C^n,C)$ such that $\varphi=BA^{-1}$, where $A:\mathbb D\to \mathcal L(\mathbb C^n, \mathbb C^n)$ is an outer function satisfying \eqref{}.
% \end{proposition}
% \begin{proof}
%     Let $\varphi=\frac{u}{v}$, where $u\in H^\infty(\mathcal L(\mathbb C^n, \mathbb C))$ and $v\in H^\infty$ is outer. Define $h:\mathbb D\to \mathcal L(\mathbb C^n, \mathbb C)$ by  $h(z)=I+\frac{u(z)^*u(z)}{|v(z)|^2}$, $z\in\mathbb D.$ Since $v\in H^\infty$ is outer, $\log|v|\in L^1(\mathbb T)$. This together with the fact that  $u\in H^\infty(\mathcal L(\mathbb C^n, \mathbb C))$, implies $\log(1+\frac{\|u\|^2}{|v|^2})\in L^1(\mathbb T).$ Hence
%     \begin{equation}
%         1\leqslant {\|h(z) \|}\leqslant  (1+\frac{\|u\|^2}{|v|^2}).
%     \end{equation}
%     This yields $\log^+ \|h(z)\|\in L^1(\mathbb T).$ Moreover, it is also easy to see that $\log^+ \|h(z)^{-1}\|=0.$ Hence there exists an outer function $C:\mathbb D\to \mathcal L(\mathbb C^n,\mathbb C^n)$ such that $h=C^*C$ on $\mathbb T.$ Taking $B=\frac{uA}{v}$, we obtain
%     \begin{equation}
%         A^*A+B^*B=h^{-1}+\frac{A^*u^*uA}{|v|^2}=A^*(I+\frac{u^*u}{|v|^2})A=
%     \end{equation}
% \end{proof}
\vspace{0.1in}

\noindent\textbf{Acknowledgement:}  The work of the first named  is supported by  ARG-MATRICS grant by the ANRF (File No: ANRF/ARGM/2025/001583/MTR) and INSPIRE Faculty Fellowship (DST/INSPIRE/04/2021/002555). The third named author acknowledges the support received from the Indian Institute of Technology Bombay through an Institute postdoctoral fellowship.

\vspace{0.1in}
\noindent\textbf{Data Availability:}
Data sharing is not applicable to this article, as no datasets were generated or analyzed during
the current study. In case any datasets are generated during and/or analyzed during the current study, they must
be available from the corresponding author on reasonable request.

\vspace{0.1in}
\noindent\textbf{Conflict of interest:} There is no conflict of interest.

\end{document}